\documentclass{amsart}
\usepackage{amssymb,enumerate,stmaryrd}

\usepackage[scr=boondox]  
           {mathalpha}
           
\usepackage{tikz-cd}
\usepackage{hyperref}
\hypersetup{%
  bookmarksnumbered=true,%
  colorlinks=true,%
  linkcolor=blue,%
  citecolor=blue,%
  filecolor=blue,%
  menucolor=blue,%
  urlcolor=blue,%
  bookmarksopen=true,%
  bookmarksdepth=2,%
  pageanchor=true}

\usepackage{mathtools}
\usepackage{todonotes}

\numberwithin{equation}{section}

\theoremstyle{plain}

\newtheorem{theorem}{Theorem}[section]

\newtheorem{lemma}[theorem]{Lemma}
\newtheorem{corollary}[theorem]{Corollary}

\theoremstyle{definition}

\newtheorem*{ack}{Acknowledgements}

\theoremstyle{remark}
\newtheorem{remark}[theorem]{Remark}

\newtheorem{chunk}[theorem]{}

\newcommand{\agr}{\operatorname{gr}}
\newcommand{\bbp}{\mathbb{P}}
\newcommand{\bp}{\operatorname{Bl}}

\newcommand{\dcoh}[1]{\operatorname{D}^b(\operatorname{coh} #1)}
\newcommand{\depth}{\operatorname{depth}}
\newcommand{\fm}{\mathfrak{m}}
\newcommand{\hh}{\operatorname{H}}
\newcommand{\length}{\operatorname{length}}

\newcommand{\mcal}[1]{\mathcal{#1}}
\newcommand{\mcf}{\mcal F}
\newcommand{\mcg}{\mcal G}
\newcommand{\mcl}{\mcal L}
\newcommand{\mco}{\mcal O}

\newcommand{\proj}{\operatorname{Proj}}
\newcommand{\pull}[1]{\mathbf{L}{#1}^*}
\newcommand{\push}[1]{\mathbf{R}{#1}_*}
\newcommand{\rees}{\mathscr{R}}
\newcommand{\sring}[1]{\operatorname{S}(#1)}
\newcommand{\spec}{\operatorname{Spec}}
\newcommand{\sspec}{\operatorname{\mathbf{Spec}}}

\begin{document}

\title[No Ulrich modules]{Non-existence of Ulrich modules over Cohen-Macaulay local rings}
\author[Iyengar]{Srikanth B. Iyengar}
\address{
Department of Mathematics\\
University of Utah\\ 
Salt Lake City, UT 84112\\ 
U.S.A.}

\author[Ma]{Linquan Ma}
\address{Department of Mathematics\\
Purdue University\\
West Lafayette, IN 47907\\
U.S.A.}

\author[Walker]{Mark E.~Walker}
\address{Department of Mathematics\\
University of Nebraska\\
Lincoln, NE 68588\\
U.S.A. }

\author[Zhuang]{Ziquan Zhuang}
\address{Department of Mathematics\\
Johns Hopkins University\\
Baltimore, MD 21218\\
U.S.A.}

\begin{abstract}
      Over a Cohen-Macaulay local ring, 
      the minimal number of generators of a maximal Cohen-Macaulay module is bounded above by its multiplicity.
      In 1984 Ulrich asked whether there always exist modules for which equality holds; such modules are known nowadays as Ulrich modules. We answer this question in the negative by constructing families of two dimensional Cohen-Macaulay local rings that have no Ulrich modules. Some of these examples are Gorenstein normal domains; others are even complete intersection domains, though not normal. 
     \end{abstract}

\date{\today}



\maketitle

\section{Introduction}
This work concerns the question of existence of certain objects over a commutative noetherian local ring $R$. A prime example of such an $R$ is the local ring of functions at a point on an algebraic variety.  In this case the structural properties of $R$ reflect the geometric properties of the variety at that point. For simplicity of exposition, we assume the ring $R$ is a domain, which in the geometric context corresponds to varieties that are irreducible at the point represented by the variety. As for groups, one can study $R$ through its representations, which in this context are called modules.  In what follows we focus on a class of special class of modules called maximal Cohen-Macaulay modules, and assume that  $R$ itself is Cohen-Macaulay; this ensures that there are many maximal Cohen-Macaulay modules. When  $R$ is a regular local ring (which is the analogue of a smooth point on a variety), maximal Cohen-Macaulay modules are all free, and in fact this property characterizes regularity. So the failure of maximal Cohen-Macaulay modules to be free is one measure of the singularity of the ring $R$.

Two of the most basic invariants attached to any module $M$ are the minimal number of elements needed to generate it, which is denoted $\nu_R(M)$, and the largest cardinality of a linearly independent set, which is called its rank and denoted $\mathrm{rank}_R(M)$. Clearly $\mathrm{rank}_R(M)\le \nu_R(M)$, but in general there is no a priori bound on the gap between these invariants. However, when $M$ is a maximal Cohen-Macaulay module, Ulrich~\cite{Ulrich:1984} proved that there is an inequality $\nu_R(M)\le e(R)\cdot \mathrm{rank}_R(M)$. Here $e(R)$ is the Hilbert-Samuel multiplicity of $R$, which is one measure of singularity of $R$.  For instance,  regular local rings $R$ are characterized by the property that $e(R)=1$.

Ulrich asked whether there always exist a module $M$ for which $\nu_R(M)= e(R)\cdot \mathrm{rank}_R(M)$; such modules were subsequently named \emph{Ulrich modules}. 
It follows from the discussion above that $R$ itself is an Ulrich module precisely when $R$ is regular. Inspired by Ulrich’s work, in 2003 Eisenbud and Schreyer \cite{Eisenbud/Schreyer:2003} introduced an analogous notion for projective varieties in algebraic geometry, which they called \emph{Ulrich sheaves}.  When $R$ has an Ulrich module, one can often reduce problems over $R$ to the case of regular rings, where they are more tractable. In the almost four decades since Ulrich modules were introduced, it was discovered that many open conjectures in commutative algebra and algebra geometry are simple consequences of the existence of Ulrich modules and sheaves. In the meantime, Ulrich modules and sheaves have been constructed in many special cases, under various algebraic and geometric assumptions.

In this paper we construct Cohen-Macaulay rings of dimension two and higher, some of which are even Gorenstein and normal, and others that are complete intersections, that possess no Ulrich modules.  Our technique for producing such examples has a geometric flavor, as it involves a relationship between Ulrich modules on $R$ and Ulrich sheaves on a projective variety  naturally associated to $R$, namely, the exceptional fiber of the blow-up of $R$ at its maximal ideal. Ulrich sheaves that arise thus are special in that they can be extended to the entire blow-up. 

While this paper settles Ulrich's question in the negative on the algebraic side, the existence of Ulrich sheaves remains an open and interesting question. 

\subsection*{A more detailed introduction}
An \emph{Ulrich module} over an arbitrary  commutative noetherian local ring $R$ is a nonzero maximal Cohen-Macaulay $R$-module $U$ with the property that its minimal number of generators, $\nu_R(U)$,  is equal to its Hilbert-Samuel multiplicity, $e_R(U)$. In \cite{Brennan/Herzog/Ulrich:1987} such modules are referred to as ``maximally generated" maximal Cohen-Macaulay modules because the inequality $\nu_R(M) \leq e_R(M)$ holds for any maximal Cohen-Macaulay $R$-module $M$. 

The mere existence of an Ulrich $R$-module has striking implications. For instance, it implies that Lech's conjecture holds for $R$: For any  flat local map $R \to S$ of local rings, one has $e(R) \leq e(S)$; see \cite[pg.~69]{Hanes:1999}, or \cite[Theorem~2.11]{Ma:2023a}.  Moreover, if $R$ admits an Ulrich module with the additional property that its class in the Grothendieck group coincides with the class of a free module, then for any $R$-module $L$ of finite length and finite projective dimension, one has $\length_R(L)\geq e(R)$; see \cite{Iyengar/Ma/Walker:2022} for more general results.

The existence of Ulrich modules is known in the following cases.
\begin{enumerate}[\quad$\bullet$]
    \item When $\dim R \leq 1$. For $\dim(R) = 0$, the residue field $R/\fm$ is Ulrich and for $\dim R = 1$, the ideal $\fm^{e(R)-1}$ is Ulrich as an $R$-module \cite[2.1]{Brennan/Herzog/Ulrich:1987}.
\item Hypersurfaces and, more generally, local rings $R$ such that  $\agr_{\fm}(R)$ is a complete intersection \cite{Herzog/Ulrich/Backelin:1991}.
\item 
When $\dim R\le 2$ and $R$ is a standard graded ring over a field; the Cohen-Macaulay case is in \cite{Brennan/Herzog/Ulrich:1987} and the general case follows from \cite{Eisenbud/Schreyer:2003}.
\item Complete local rings $R$ with $\dim R\le 2$ and $\proj(\agr_\fm(R))$ geometrically reduced. This result will appear in forthcoming work by the first three authors. 
\end{enumerate}
See Emre Coskun's article~\cite{Coskun:2017} for a nice survey on Ulrich modules and Ulrich sheaves.

Ulrich asked \cite[\S 3]{Ulrich:1984} whether every Cohen-Macaulay local ring admits an Ulrich\footnote{He did not refer to them as ``Ulrich modules".} module. Yhee~\cite{Yhee:2023} constructed  local domains of dimension two that have no Ulrich modules. However, these examples are not Cohen-Macaulay. 

In this work, we answer Ulrich's original question in the negative:

\begin{theorem} 
\label{th:intro} 
There exist Gorenstein normal domains, and complete intersection local domains, of dimension two that have no Ulrich modules.
\end{theorem}

This result packages Theorems~\ref{th:ci} and \ref{th:gorenstein}. We do not know if the rings we construct admit lim Ulrich sequences, in the sense of \cite{Ma:2023a}.

Let $B$ be the blowup of $\spec(R)$ at its closed point and $E = \proj(\agr_\fm(R))$ the exceptional fiber. As mentioned, the existence of an Ulrich module on $R$ leads
to an Ulrich sheaf on $E$ that extends to $B$.  Our technique is to find examples where there are no extendable Ulrich sheaves on $E$. 
All of our examples can be obtained by starting with an ample line bundle $\mcl$ on a projective curve $C$ such that $\Gamma(C, \mcl) = 0$ and the canonical pairing
\[
\Gamma(C, \mcl^2) \otimes_k \Gamma(C, \mcl^j) \longrightarrow \Gamma(C, \mcl^{j+2})
\]
is a surjection for all $j\ge 2$. In this situation, 
the local ring  $R$ obtained by localizing (or completing) the section ring $\bigoplus_{j \geq 0} \Gamma(C, \mcl^j)$ at its homogeneous maximal ideal,  cannot have any Ulrich modules; see Corollary~\ref{co:noU-geometric}. 

The simplest situation in which this occurs is when $C$ is a smooth quartic plane curve and $\mcl$ is a non-effective theta characteristic. There are $36$ choices for such an $\mcl$ (assuming $k$ is algebraically closed of characteristic not $2$), and each choice leads to a complete Gorenstein normal domain of dimension two without any Ulrich modules. The complete intersection  example in Theorem \ref{th:intro} arises from the singular projective plane curve $C$ cut out by the polynomial $(y^3+x^2z)^2-x^3z^3$ and a line bundle $\mcl$ satisfying $\mcl^2 \cong \mco_C(1)$.

\section{Ulrich modules and Ulrich sheaves}
\label{se:basic}
Let $(R,\fm,k)$ be a noetherian local ring with maximal ideal $\fm$ and residue field $k$. 
Let $M$ be a nonzero finitely generated $R$-module.  The Krull dimension and depth of $M$ are denoted $\dim_R M$ and $\depth_R(M)$, respectively. The $R$-module $M$ is \emph{maximal Cohen-Macaulay} if it is finitely generated and $\depth_R(M)= \dim R$. We write $\nu_R(M)$ for the minimal number of generators of $M$, and $e_R(M)$ for its Hilbert-Samuel multiplicity with respect to the ideal $\fm$. 

\subsection*{Ulrich modules}
Any maximal Cohen-Macaulay $R$-module $M$ satisfies $\nu_R(M) \le e_R(M)$; see \cite[Section~3]{Ulrich:1984}. When $e_R(M)= \nu_R(M)$ and $M$ is nonzero it is said to be an \emph{Ulrich module}.

Assume $k$ is infinite and let $J\subseteq \fm$ be a minimal reduction of $\fm$. For every maximal Cohen-Macaulay $R$-module $M$, we have 
\[
e(\fm, M) = e(J, M) = \ell_R(M/JM) \geq \ell_R(M/\fm M)=\nu_R(M)\,.
\] 
Thus a nonzero maximal Cohen-Macaulay module $M$ is Ulrich if and only if $J M= \fm M$ for one (or equivalently, every) minimal reduction $J$ of $\fm$.

\subsection*{Ulrich sheaves}
Let $X$ be a projective scheme over a field and let $\mco_X(1)$ be a very ample line bundle on $X$. A coherent sheaf $\mcf$ is called an \emph{Ulrich sheaf for $(X,\mco_X(1))$} if it is nonzero and  one of the following equivalent conditions holds:
\begin{enumerate}
    \item $\hh^i(X, \mcf(j))=0$ for all $-\dim X\le j \le -1$, and all $i$.  
    \item $\hh^i(X, \mcf(j))=0$ except when $i=0$, $j\geq 0$ or $i=\dim X$, $j\leq -\dim X-1$.
\end{enumerate}
Suppose $\pi\colon X\to \bbp^n_k$ is a finite map such that $\pi^*\mco_{\bbp^n_k}(1)=\mco_X(1)$; such a map exists when $k$ is infinite. Then a coherent sheaf $\mcf$ on $X$ is Ulrich if and only if $\pi_*\mcf \cong \mco_{\bbp^n_k}^r$ for some integer $r\ge 1$;
see Eisenbud and Schreyer~\cite{Eisenbud/Schreyer:2003}.

The definition of an Ulrich sheaf depends on the very ample line bundle $\mco_X(1)$. We  suppress $\mco_X(1)$ when it is clear from the context. For instance, when $R$ is a standard graded noetherian ring over a field and $X=\proj(R)$, we take $\mco_X(1)$ to be the coherent sheaf associated to $R(1)$.

\begin{remark}
\label{re:ulrich-dim-1}
If $\dim X=1$, then $\mcf$ is an Ulrich sheaf if and only if $\hh^i(X,\mcf(-1))=0$ for all $i$ and  $\mcf\ne 0$.  Thus the existence of an Ulrich sheaf for $(X, \mco_X(1))$ is equivalent to the existence of a nonzero coherent sheaf on $X$ with no cohomology, and the latter property does not depend on the choice of $\mco_X(1)$.
\end{remark}

\subsection*{Rees algebras and blowups}
Let $R$ be a local ring and $M$ a finitely generated $R$-module. Consider the graded $R$-module
\[
\rees_\fm(M) \coloneqq \bigoplus_{n\geqslant 0}(\fm^n M )t^n\,.
\]
Then $\rees_\fm(R)$ is a finitely generated graded $R$-algebra, called the \emph{Rees ring} of $R$ with respect  $\fm$, and $\rees_\fm(M)$ is a finitely generated graded $\rees_\fm(R)$-module.  The canonical map $R\to\rees_\fm(R)$ induces the morphism 
\[
p\colon \bp_\fm(R)\coloneqq \proj (\rees_\fm(R)) \longrightarrow \spec R\,.
\]
This is the \emph{blowup} of $\spec R$ at $\{\fm\}$. The \emph{strict transform} of an $R$-module $M$ is the coherent sheaf on $\bp_\fm(R)$ associated to the $\rees_\fm(R)$-module $\rees_\fm(M)$. 

The \emph{associated graded} module of $M$ is
\[
\agr_{\fm}(M) = k\otimes_R \rees_\fm(M) = \bigoplus_{n\geqslant 0}(\fm^n M/\fm^{n+1}M)t^n\,.
\]
Thus $\agr_{\fm}(R)$ is a finitely generated standard graded $k$-algebra and the $\agr_\fm(R)$-module $\agr_\fm(M)$ is finitely generated by ${\agr_{\fm}(M)}_0$.

Set $E= \proj (\agr_\fm(R))$; this is the exceptional fiber of the blowup $p$.
\[
\begin{tikzcd}
E \arrow[d,"q" swap] \arrow[r, hookrightarrow,"j"] & \bp_\fm(R) \arrow[d,"p"] \\
\spec k \arrow[r, hookrightarrow] &\spec R
\end{tikzcd}
\]
The strict transform of $M$ can also be described as the quotient $p^*M/\Gamma_E (p^*M)$. 
 
\begin{lemma} 
\label{le:modules-sheaves}
Let $U$ be an Ulrich $R$-module and $\mcf$ be the strict transform of $U$. Then $\pull j\mcf \cong j^*\mcf$ in $\dcoh{E}$ and $j^*\mcf$ is an Ulrich sheaf on $E$.
\end{lemma}

\begin{proof}
Since $U$ is an Ulrich $R$-module, the graded $\agr_\fm(R)$-module $\agr_\fm(U)$ is Ulrich, in that it is maximal Cohen-Macaulay, generated in degree $0$, and its multiplicity is equal to the number of generators. This is proved in \cite[Corollary~1.6]{Brennan/Herzog/Ulrich:1987} under additional assumptions on $R,U$.  Here is an argument in the general case.

The multiplicity and minimal number of generators of $\agr_\fm(U)$ as a module over $\agr_\fm(R)$ coincide with the corresponding invariants of the $R$-module $U$. Thus the crucial point is that $\agr_\fm(U)$ is maximal Cohen-Macaulay.  We can assume the residue field of $R$ is infinite, and hence that there exists a $U$-regular sequence, say $\boldsymbol{x}\coloneqq x_1,\dots,x_d$, such that $\boldsymbol{x} U=\fm U$. Since $\boldsymbol{x}$ is $U$-regular, one has an isomorphism of $\agr_{(\boldsymbol{x})}(R)$-modules
\[
(U/\boldsymbol{x}U)[X_1,\dots,X_d]\xrightarrow{\ \cong \ } \agr_{(\boldsymbol{x})}(U) \quad \text{with $X_i\mapsto \overline{x}_i$.}
\]
Here $\overline{x}_i$ denotes the residue class of $x_i$ in $(\boldsymbol{x})/(\boldsymbol{x})^2$. Therefore $\agr_{(\boldsymbol{x})}(U)$ is a maximal Cohen-Macaulay module over $\agr_{(\boldsymbol{x})}(R)$. It remains to observe that $\agr_{(\boldsymbol{x})}(U)$ is isomorphic to $\agr_\fm(U)$, since $\boldsymbol{x} U=\fm U$.

It follows that $j^*\mcf$, which is the  coherent sheaf associated to $\agr_\fm(U)$, is an Ulrich sheaf; see \cite[Proposition~2.1]{Eisenbud/Schreyer:2003}. The coherent sheaf $\mcf (1)$ is the sheaf associated to  $\rees_{\fm}(\fm U)$ and the map $\mcf(1) \to \mcf$ induced by the natural map $\mco_B(1) \to \mco _B$, 
where $B=\bp_\fm(R)$, arises from the canonical inclusion 
\[
 \rees_\fm(\fm U) = \bigoplus_{i \geqslant 0} \fm^{i+1}U \subseteq \bigoplus_i \fm^i U = \rees_\fm(U)\,.
  \]
In particular, the map $\mcf(1) \to \mcf $ is injective. This proves that $\pull j\mcf \cong j^*\mcf$.
\end{proof}

In the context of the Lemma~\ref{le:modules-sheaves}, one has also  $\push\pi\mcf\cong U$. Moreover, the result has a converse, and this can be used to construct Ulrich modules over certain two dimensional rings. These results will appear in forthcoming work by the first three authors (we will not use them in the sequel though).

\section{Criteria for non-existence}
In this section we identify obstructions to the existence of an Ulrich module over a local ring in terms of the geometry of the blowup at its maximal ideal. 

\begin{theorem}
\label{th:noU}
Let $R$ be a noetherian ring and $\fm\subset R$ a maximal ideal with $\dim R_\fm\ge 2$. Let $\bp_\fm(R)\to \spec R$ be the blowup at $\{\fm\}$ and $E$ the exceptional fiber of the blowup. If there exists a locally principal ideal sheaf  ${\mcal I}$ on $\bp_\fm(R)$  and an integer $n\ge 2$ such that ${\mcal I}^n$  defines $E$ in $\bp_\fm(R)$, then neither the local ring $R_\fm$ nor its $\fm$-adic completion has any Ulrich modules. 
\end{theorem}

\begin{proof}
Let $\widehat{R}_\fm$ denote the $\fm$-adic completion of $R_\fm$. Since the maps $R\to R_\fm$ and $R\to \widehat{R}_\fm$ are flat, it is easy to verify that the hypotheses carry over to both $R_\fm$ and $\widehat{R}_\fm$. Thus we can assume in the rest of the argument that $R$ is local, with maximal ideal $\fm$, and that $\dim R\ge 2$. 

In what follows we write $B$ instead of $\bp_\fm(R)$. Observe that ${\mcal I}^n\cong \mco_{B}(1)$. Let $W$ be the subscheme of $B$ cut out by $\mcal I$. 

We first consider the case when $\dim R=2$; equivalently, $\dim W=1$.

Assume to the contrary that there exists an Ulrich $R$-module $U$.  Let $\mcf$ be the strict transform of $U$ and $\mcal U$ the sheaf on $E$ associated to $\agr_\fm(U)$.  We may identify $\mcal U$ with $\mcf/ {\mcal I}^n \mcf$. Setting ${\mcal U}_i$ to be the coherent sheaf on $E$ given by ${\mcal I}^i \mcf/{\mcal I}^n \mcf$, we obtain a filtration 
\[
0 = \mcal U_n \subseteq \cdots \subseteq \mcal U_0 = \mcal U
\]
of coherent sheaves on $E$. Identifying the sub-quotients with their restrictions to $W$, for each $i$ one gets
\[
{\mcal U}_{i}/{\mcal U}_{i+1} \cong  \mcal U|_{W} \otimes_{\mco_W} \mcl^i
\]
where $\mcl =\mcal I/{\mcal I}^2$ is the line bundle that is the restriction of $\mcal I$ to $W$. Observe that $\mcl$ is ample because $\mcl^n\cong {\mcal I}^n/{\mcal I}^{n+1}$ is the restriction of $\mco_B(1)$ to $W$. 

Set $\mcg \coloneqq ({\mcal U}_0/{\mcal U}_{1}) (-1)$, again viewed as a sheaf on $W$. Since $\dim W=1$ one has that $\dim \mcg=1$ as well and hence the function
 \[
 \chi(\mcg \otimes \mcl^j)= \mathrm{rank}_k \hh^0(W, \mcg \otimes \mcl^{j})
    -\mathrm{rank}_k \hh^1(W, \mcg \otimes \mcl^{j})
 \] 
is strictly increasing  in $j$. Indeed $\chi(\mcg \otimes \mcl^j)$ is a polynomial function in $j$ of degree at most one, by ~\cite[Lemma~33.4.5.1]{stacks-project}, and since $\mcl$ is ample and $\dim\mcg=1$, this polynomial must be linear with positive slope. See \cite[Propositions~23.78, 23.79]{Gortz/Wedhorn:2010} for more precise results.

The filtration of $\mcal U$ above gives an injection and a surjection:
\[
\mcg \otimes \mcl^{n-1} \hookrightarrow {\mcal U}(-1) \quad \text{and}\quad 
    {\mcal U}(-1) \twoheadrightarrow \mcg\,.
\]
Since $\mcal U$ is an Ulrich sheaf, ${\mcal U}(-1)$ has no cohomology. Combined with the fact that $\dim W = 1$ one gets
\[
\hh^0(W, \mcg \otimes \mcl^{n-1}) = 0 = \hh^1(W, \mcg)
\]
and therefore that
\[
\chi(\mcg \otimes \mcl^{n-1}) \leq 0 \leq \chi(\mcg)\,.
\]
Since $n\ge 2$, this contradicts the fact that $\chi(\mcg \otimes \mcl^j)$ is strictly increasing.

This completes the proof when $\dim R=2$.

It remains to explain the reduction to this case. So suppose $\dim R\ge 3$ and let $U$ be an Ulrich $R$-module. By enlarging the residue field if needed we can assume it is infinite. One can then find a system of parameters $\boldsymbol{x} \coloneqq x_1,\dots, x_d$ of $R$ that is a regular sequence on $U$ and such that $\boldsymbol{x}U = \fm U$. Let 
\[
R' = R/(x_1,\dots, x_{d-2}) \quad\text{and}\quad  U' = U/(x_1,\dots, x_{d-2})U.
\]
It is clear that the $R'$-module $U'$ is Ulrich.  So it suffices to show the hypotheses still hold for the ring $R'$. 

Let $B$ and $B'$ denote the blowups of $\spec R$ and $\spec R'$ at their respective closed points, and let $E$ and $E'$ be the corresponding exceptional fibers.  It is straightforward to verify that the following diagram of rings is co-cartesian:
\[
\begin{tikzcd}
\rees_\fm(R) \arrow[d,twoheadrightarrow]\arrow[r,twoheadrightarrow] 
    & \rees_{\fm R'}(R') \arrow[d,twoheadrightarrow] \\
\agr_\fm(R) \arrow[r,twoheadrightarrow] 
    & \agr_{\fm R'}(R')
\end{tikzcd}
\]
 It thus follows that the square below is cartesian:
\[
\begin{tikzcd}
E' \arrow[d,hookrightarrow]\arrow[r,hookrightarrow] 
    & E \arrow[d,hookrightarrow] \\
B'\arrow[r,hookrightarrow]
    & B
\end{tikzcd}
\]
Our hypothesis is that the ideal $\mcal J$ cutting $E$ out of $B$ satisfies $\mcal J = \mcal I^n$ for some locally principal sheaf of ideals $\mcal I$ on $B$ and integer $n \geq 2$. Let $\mcal I'$  and $\mcal J'$ denote the pull-back of $\mcal I$ and $\mcal J$, respectively,  to $B'$. Since the square above is cartesian, $\mcal J'$ coincides with the ideal cutting out $E'$ from $B'$ and hence it is locally generated by a non-zero-divisor. Since $\mcal J = \mcal I^n$, we have $\mcal J' = (\mcal I')^n$, and since $\mcal I$ is locally principal, so is $\mcal I'$. Thus, all the assumptions are valid for $R'$, as desired. 
\end{proof}

Here is a special case of Theorem~\ref{th:noU}.

\begin{corollary}
\label{co:noU}    
Let $(R,\fm,k)$ be a noetherian local ring such that $\dim R\ge 2$. Let $ \bp_\fm(R) \to \spec R$ be the blowup of $\{\fm\}$ and $E$ the exceptional fiber of the blowup. If $\bp_\fm(R)$ is locally factorial (for instance, regular) and $E$  is irreducible but not reduced, then  $R$ has no Ulrich modules. 
 \end{corollary}

\begin{proof}
Set $W\coloneqq E_{\mathrm{red}}$. The hypothesis implies that it is an integral scheme and hence an integral Weil divisor on $\bp_\fm(R)$. Since $\bp_\fm(R)$ is locally factorial, we know that $W$ is a Cartier divisor and $[E]=n[W]$ for some $n\ge 2$, since $E$ is not reduced. Thus Theorem~\ref{th:noU} applies.
\end{proof}

Our interest in the formulation above is that it applies also when $R$ is not a  section ring; see Remark~\ref{re:Yhee} where we revisit the example constructed by Yhee~\cite{Yhee:2023}.

Next we identify another source of rings satisfying the hypotheses of Theorem~\ref{th:noU}.

\subsection*{Section rings}
The section ring of a line bundle $\mcl$ over a scheme $V$ is:
\[
\sring{V,\mcl} \coloneqq \bigoplus_{j\geqslant 0}\Gamma(V, \mcl^j)\,.
\]
When $k$ is a field, $V$ is a proper $k$-scheme, and $\mcl$ is ample,  $\sring{V,\mcl}$ is a finitely generated graded $k$-algebra.

\begin{corollary}
\label{co:noU-geometric}
Let $V$ be a projective $k$-scheme with $\dim V\ge 1$, let $\mcl$ be an ample line bundle on $V$, and set $S = \sring{V,\mcl}$. Suppose there is an integer $a \geq 2$ so that the following conditions hold:
\begin{enumerate}[\quad \rm (1)]
    \item $S_0 = k$ and $S_j = 0$ for $1 \leq j \leq a-1$;
    \item the multiplication map $S_a \otimes_k S_j \to S_{a+j}$ is surjective for all $j \geq a$.  
\end{enumerate}  
Then for $\fm=S_{\geqslant 1}$, neither the local ring $S_\fm$ nor its $\fm$-adic completion admits an Ulrich module. 
\end{corollary}

\begin{proof}
We claim that as schemes over $\spec S$ one has 
\[
\bp_\fm(S) \cong \sspec_V (\mathrm{Sym}(\mcl)) \quad \text{where $\mathrm{Sym}(\mcl)$ is the sheaf of algebras $\bigoplus_{i\geqslant 0} \mcl^i$.}
\] 
To verify this, we first verify by an induction on $j$ that  $\fm^j = S_{\geqslant ja}$ for all $j \geq 0$. 

The base case $j=1$ is clear. Assuming the desired equality holds for some $j\ge 1$, one gets equalities
\[
\fm^{j+1}=\fm \fm^j = S_{\geqslant a}\cdot S_{\geqslant aj} = S_{\geqslant (j+1)a}
\]
where the second one is by the induction hypothesis, and the last one is immediate from condition (2) in the statement.

The equality $\fm^j=S_{\geqslant ja}$ justifies the first isomorphism below:
\begin{align*}
    \bp_{\fm}(S) 
    \cong \proj\big(\bigoplus_{j\geqslant 0} S_{\geqslant ja} t^j \big) 
    \cong \proj\big(\bigoplus_{j\geqslant 0} S_{\geqslant j} t^j \big)
    \cong \sspec_V(\mathrm{Sym}(\mcl))\,,
\end{align*}
The second one holds as  $\oplus_{j \geqslant 0} S_{\geqslant ja} t^j$ is the $a$-th Veronese subring of $\oplus_{j\geqslant 0} S_{\geqslant j} t^j$, and  the last isomorphism holds because $\mcl$ is ample; see \cite[Section~8.7.3]{EGA2:1961} or \cite[Paragraph~6.2.1]{Hyry/Smith:2003}. These isomorphisms are compatible with projections to $\spec S$. 

This justifies the claim about the blowup.

The exceptional fiber $E$ of the blowup $\bp_\fm(S) \to \spec S$ is the  subscheme of $\bp_\fm(S)$ cut out by $\mcl^a$, that is to say 
\[
E\cong \sspec_V\left(\frac{\mathrm{Sym}(\mcl)}{\mcl^a \, \mathrm{Sym}(\mcl)} \right)\,.
\]
As $V$ is the subscheme of $\bp_\fm(S)$ cut out by $\mcl$ and $a\ge 2$, we can apply  Theorem~\ref{th:noU} to get the desired result.
\end{proof}

\begin{remark}
  Conditions (1) and (2) imply that $S$ is generated as a $k$-algebra by its components in degrees $[a, 2a-1]$. However the latter condition is weaker: the $k$-algebra $k[x,y]$ with $|x|=2$ and $|y|=3$ is generated by its components in degrees $[2,3]$, but the map $S_2\otimes_k S_4\to S_6$ is not surjective because $y^2$ is not in the image.
  
  \end{remark}

Corollary~\ref{co:noU-geometric} brings up the question: When is a noetherian graded ring over a field the section ring of an ample line bundle on a projective scheme? Here is one answer; it is extrapolates from \cite[Proposition~2.1]{Hyry/Smith:2004}, where the result is stated for normal domains. The argument does not require these conditions, and we sketch a proof, to convince ourselves. 

\begin{lemma}
\label{le:section-ring}
Let $k$ be a field and $S$ a noetherian $\mathbb{N}$-graded $k$-algebra with $S_0=k$ and $\depth S\ge 2$. Then $S$ is  the section ring of an ample line bundle on a projective scheme over $k$ if and only if there exist homogeneous elements $x_0,\dots,x_n$ in $S$ such that $\sqrt{(x_0,\dots,x_n)}=S_{\geqslant1}$ and the $\mathbb{Z}$-graded rings $S_{x_i}$ have a unit of degree $1$.
\end{lemma}

\begin{proof}
Suppose $S=\sring{X,\mcl}$. Choose an integer $m\gg 0$ such that  $\mcl^m$ and $\mcl^{m+1}$ are very ample (see \cite[Theorem II, 7.6]{Hartshorne:1977} and \cite[Theorem 1.2.6]{Lazarsfeld:2004}) and let $s_1,\dots,s_r$ and  $t_1,\dots,t_l$ be bases of $\Gamma(X,\mcl^m)$ and $\Gamma(X,\mcl^{m+1})$, respectively. One can check that the radical of the ideal generated by $x_{ij}=s_i t_j$, for $1\le i\le r$ and $1\le j\le l$, is $S_{\geqslant1}$ and that $t_j^2/x_{ij}$ in $S_{x_{ij}}$ is a unit of degree 1. 

For the converse, by computing global sections using the \v{C}ech cover associated
to the sequence $\boldsymbol{x} = x_0, \dots, x_n$, one gets an exact sequence
\[
0\longrightarrow \hh_{\fm}^0(S)\longrightarrow 
    S\xrightarrow{\ \sigma\ }\bigoplus_{j\geqslant 0}\hh^0(\proj S; \widetilde{S(j)}) 
        \longrightarrow \hh_{\fm}^1(S)_{j\geqslant 0} \longrightarrow 0\,,
\]
where $\fm\coloneqq S_{\geqslant 1}$, the homogeneous maximal ideal of $S$. Since $\depth S\ge 2$, we deduce that the map $\sigma$, which is a map of graded $k$-algebras, is an isomorphism. The hypothesis on $\boldsymbol{x}$ translates to the existence of an open cover 
\[
\proj S= D(x_0)\cup \cdots \cup D(x_n)
\]
such that one has an isomorphism of $\mathbb{Z}$-graded rings $S_{x_i}\cong [S_{x_i}]_0[t,t^{-1}]$, where $t$ has degree $1$. Using this it is easy to check that the natural map
\[
\widetilde{S(i)}\otimes \widetilde{S(j)} \longrightarrow \widetilde{S(i+j)}
\]
is an isomorphism for all integers $i,j$. Hence $\mcl\coloneqq \widetilde{S(1)}$ is invertible and $\widetilde{S(j)}\cong \mcl^{j}$. It follows that $S$ is the section ring of the ample line bundle $\mcl$ on $\proj S$.
\end{proof}

\begin{remark}(Yhee's example revisited)
\label{re:Yhee}
     Let $k$ be a field,  $S=k[\boldsymbol{x}]$ the polynomial ring over $k$ in indeterminates $\boldsymbol{x}=x_1,\dots,x_d$, and set $\fm=(\boldsymbol{x})$. Fix an integer $n\ge 2$,  set $J=\fm^n$ and $R=k\oplus J$, viewed as a subring of $S$. Then $R$ is not a section ring.

We claim that the ring $R$ satisfies the hypothesis of Corollary~\ref{co:noU}.

Indeed, the cokernel of the inclusion $\rees_J(R)\subseteq \rees_J(S)$  is $S/R$, and in particular finite dimensional as a $k$-vector space. Thus one gets the first isomorphism below:
\[
\bp_J(R) \xleftarrow{\ \cong\ }\bp_J(S) \cong \bp_\fm(S) \cong \bbp^{d-1}_{S}\,.
\] 
The second isomorphism holds as $\rees_J(S)$ is the $n$th Veronese subring of $\rees_\fm(S)$, since $J=\fm^n$.  In particular $\bp_J(R)$ is regular.  The exceptional fiber of the blowup is $E = \proj(\agr_J(R))$ where 
\[
\agr_J(R)=k\oplus \frac{\fm^{n}}{\fm^{2n}} \oplus \frac{\fm^{2n}}{\fm^{3n}}\oplus \cdots .
\]
Observe that the ideal sheaf, say $\mcal I$, corresponding to $\fm^{n+1}\agr_J(R)$ is nilpotent, of order $n$, and  nonzero; the latter holds because $n\ge 2$. So $E$ is not reduced. It is irreducible because $E/\mcal I$ is the projective scheme associated to the $k$-algebra
\[
k\oplus \bigoplus_{i\geqslant 1} \frac{\fm^{ni}}{\fm^{ni+1}},
\]
which is isomorphic to the $n$th Veronese subring of $S$, so $E/\mcal I\cong \bbp^{d-1}_k$.
\end{remark}

\section{Examples: Complete intersections}
\label{se:ci}
In this section we construct complete intersection local rings, some of which are domains, that have no Ulrich modules. Let $k$ be a field and consider the $k$-algebra
\[
S\coloneqq \frac{k[s,t,x,y,z]}{(s^2-x^3,st-f(x,y,z),t^2-z^3)}
\]
where $f(x,y,z)$ is a homogeneous polynomial in $k[x,y,z]$  of the form 
\[
f(x,y,z)=y^3+ f_2(x,z)y^2+f_1(x,z)y+f_0(x,z)
\]
where each $f_i(x,z)$ is in $k[x,z]$. Set the degree of $x,y,z$ to $2$ and that of $s,t$ to $3$. Then $S$ is an $\mathbb{N}$-graded ring, with homogeneous maximal ideal  $\fm\coloneqq S_{\geqslant 2}$.

\begin{theorem}
\label{th:ci}
For any choice of $f(x,y,z)$ satisfying the condition above neither the local ring $S_\fm$ nor its $\fm$-adic completion has an Ulrich module. They are both  complete intersections of dimension $2$.
\end{theorem}

\begin{proof}
It suffices to  verify that the graded $k$-algebra $S$ has the following properties:
\begin{enumerate}[\quad\rm(a)]
    \item 
        It is a complete intersection ring, with $\dim S=2$;
    \item 
        It satisfies the conditions (1) and (2) of Corollary~\ref{co:noU-geometric}, with $a=2$;
    \item 
        It is the section ring of an ample line bundle on a projective curve. 
\end{enumerate}

Proof of (a): It is clear that $\agr_\fm(S)$ is a quotient of the ring 
\[
\frac{k[s,t,x,y,z]}{(s^2,st,t^2,x^3z^3-f(x,y,z)^2)}\,.
\]
Since $f(x,y,z)$ is monic in $y$ we get that $\dim S\le 2$. On the other hand, $S$ is defined by three relations, so it must be a complete intersection ring, with $\dim S=2$. 

Proof of (b): Evidently $S_0=k$ and $S_1=0$. Let $R$ be  the second Veronese subring of $S$. It is clear from the relations defining $S$ that $R$ is the $k$-sub-algebra of $S$ generated by $x,y,z$. As an $R$-module, $S$ is generated by $1,s,t$. It follows that $S_2\cdot S_j=S_{j+2}$ for $j\ge 2$.
 
 Proof of (c): We use the criterion from Lemma~\ref{le:section-ring}. The elements $x,z$ form a homogeneous system of parameters for $S$. The element $s/x$ is of degree $1$ in $S_x$, and invertible, with inverse $s/x^2$. By symmetry, $t/z$ is a unit of degree $1$ in $S_z$.

This completes the proof of the theorem.
\end{proof}

The preceding theorem provides a family of complete intersections with no Ulrich modules. Special values of $f(x,y,z)$ yield rings with further interesting properties.

\subsection*{The case $f(x,y,z)=y^3+x^2z$} Then the ring in question is a domain
\[
S\coloneqq \frac{k[s,t,x,y,z]}{(s^2-x^3,st-(y^3+x^2z),t^2-z^3)}.
\]
For $k=\mathbb{Q}$, this claim can be verified using, for instance, Macaulay2~\cite{M2}. Here is a sketch of an argument for general $k$: Since $x$ is a parameter in $S$, it is not a zero-divisor, so it suffices to check that the ring $S[x^{-1}]$ is a domain. One has
\[
S[x^{-1}] = \frac{A[x^{-1},s]}{(s^2-x^3)} \quad\text{where} \quad 
    A= \frac{k[x,y,z]}{((y^3+x^2z)^2-x^3z^3)}
\]
It thus suffices to verify that $A$ is a domain and $x^3$ does not have a square root in the fraction field of $A$.

The property that $A$ is a domain translates to the irreducibility of the polynomial $(y^3-x^2z)^2-x^3z^3$. Specializing $y=1$, it suffices to verify that the polynomial
\[
x^4z^2 - x^3z^3 - 2x^2z +1
\]
is irreducible. This can be  done easily using the Newton polytope associated to this polynomial, as  described in Gao~\cite{Gao:2001}: The polytope is a triangle with vertices
\[
v_0=(0,0), v_1=(3,3), v_2=(4,2).
\]
The greatest common divisor of the coordinates of $v_0-v_1$ and $v_0-v_2$ is $1$, so the polytope is integrally indecomposable, by \cite[Corollary~4.5]{Gao:2001}.  We conclude that polynomial in question is irreducible; see \cite[p.~507]{Gao:2001}.

Finally, if $x^3$ has a square root in the fraction field of $A$, then so is $x$. Thus there exist nonzero elements $a, b \in A$ such that $a^2x=b^2$. Since $A$ can be standard graded, with degrees of $x,y,z$ equal to $1$, this is not possible since the highest degree term in $a^2x$ has odd degree while the highest degree term in $b^2$ has even degree.

\subsection*{The case $f(x,y,z)=y^3$ and $k$ is algebraically closed}
The complete intersection in question is then
\[
S\coloneqq \frac{k[s,t,x,y,z]}{(s^2-x^3,st-y^3,t^2-z^3)}
\]
and this is closely connected to Yhee's construction~\cite{Yhee:2023} of domains (not Cohen-Macaulay) that have no Ulrich modules.

Indeed, according to Macaulay2~\cite{M2}, the ring $S$ has three minimal prime ideals. Evidently, for any of minimal prime $\mathfrak{p}$ the domain $S/\mathfrak{p}$ (or rather it localization at the maximal ideal) cannot have Ulrich modules, for any Ulrich module over it would be Ulrich over $S_\fm$. Again according to Macaulay2~\cite{M2}, one of the prime ideals is
\[
\mathfrak{p}=(s^2-x^3,st-xyz, z^3-t^2, y^2-xz, sz^2-xyt, yzs-x^2t)
\]
as an ideal in the underlying polynomial ring. The assignment 
\[
s\mapsto u^3, t\mapsto v^3, x\mapsto u^2, y\mapsto uv, z\mapsto v^2
\]
induces an isomorphism of $k$-algebras
\[
S/\mathfrak{p} \cong k[u^2,uv,v^2,u^3,v^3]\subset k[u,v].
\]
This subring differs from the subring $k+ (u,v)^2$ considered in \cite{Yhee:2023}, as discussed in the previous section, only in containing $u^2v,uv^2$, and one can argue as in \cite{Yhee:2023} (or as in the previous section) to deduce that $S/\mathfrak{p}$ does not have Ulrich modules.

The same analysis carries over to the quotient of $S$ by the other minimal primes. Summing up, the complete intersection ring $S$ above can be thought of being obtained by gluing three copies of the (non Cohen-Macaulay) ring constructed by Yhee in \cite{Yhee:2023}, as discussed in the previous section. 

\begin{remark}
More generally, consider the $k$-algebra
    \[
     S \coloneqq \frac {k[s_1,\dots,s_n,x_{ij}\mid 1\le i\le j\le n]}{(s_is_j-x_{ij}^3\mid 1\le i\le j\le n)}\,.
    \]
Then $S_{\fm}$ is a complete intersection of dimension $n$, and an argument similar to the proof of Theorem~\ref{th:ci} shows that $S_{\fm}$ does not admit Ulrich modules. Moreover the ideal $(x_{11},\dots,x_{nn})$ is a minimal reduction of $\fm$; it is easy to check the other variables are integral over this ideal. Thus 
\begin{align*}
e(S_{\fm}) 
        & =\length_S\left(\frac S{(x_{11},\dots,x_{nn})}\right) \\
        & =\length_S\left(\frac{k[s_1,\dots,s_n, x_{ij} \mid 1\le i < j \le n]}{(s_1^2,\dots,s_n^2, s_is_j-x_{ij}^3\mid 1\le i <j\le n)}\right)\\
& = 2^n \cdot 3^{\binom{n}{2}}.
\end{align*}
where the last equality follows because the ring in question is a finite free extension of $k[s_1,\dots,s_n]/(s_1^2,\dots,s_n^2)$ of degree $3^{\binom{n}{2}}$. Note that $S_{\fm}$ is not a strict complete intersection, as the product of the $\fm$-adic orders of the defining equations is $2^{\binom{n+1}{2}}$. One thus gets a family of complete intersections with no Ulrich modules, and for which the difference between the orders and multiplicity is arbitrarily large. 
\end{remark}

\section{Examples: Gorenstein normal domains}
\label{se:gorenstein}
In this section we construct Gorenstein normal domains with no Ulrich modules.

Let $C$ be a smooth non-hyperelliptic curve of genus $g\ge 3$ (for example, a smooth plane quartic) over an algebraically closed field $k$. For simplicity, we assume that char$(k)=0$. Let $\mcl$ be a non-effective theta characteristic on $C$; in other words
\[
\mcl^{\otimes 2} \cong \omega_C \quad\text{and}\quad \Gamma(C,\mcl) = 0\,.
\]
It is a classical result \cite{Mumford:1966} (see also \cite{Auffarth/Pirola/Manni:2017}) that non-effective theta characteristic exists on any smooth curve. Consider, for example, the curve
\[
\proj \frac{\mathbb{C}[x,y,z]}{(x^4+y^4+z^4)}\,.
\]
For any plane quartic, there are $2^{2g} = 64$ theta characteristics, but only $28$ of them are effective (corresponding to the $28$ bitangents); see \cite[Exercise IV, 2.3(h)]{Hartshorne:1977}.

\begin{theorem}
   \label{th:gorenstein}
Set $S\coloneqq\sring{C,\mcl}$ and let $\fm$ be the homogeneous maximal ideal of $S$. Then neither the local ring $S_\fm$ nor its completion at $\fm$ admits an Ulrich module. They are both  Gorenstein and normal.
\end{theorem}

\begin{proof}
To begin with $S$ is normal because it is the section ring of smooth projective variety, with respect to an ample invertible sheaf; see~\cite[Section~1]{Smith:1997}. It is also Gorenstein by \cite[Proposition 3.14(4)]{Kollar:2013}. These conditions are then inherited by $S_\fm$ and its completion. For the rest, we show that the conditions in Corollary \ref{co:noU-geometric} are satisfied with $a=2$. 

Indeed, since $C$ is not hyperelliptic, the canonical ring $\sring{C,\omega_C}$ is generated in degree 1 by a classical theorem of Max Noether. As $\mcl^{\otimes 2}\cong \omega_C$, we deduce that the multiplication map $\mathrm{Sym}^m(S_2)\to S_{2m}$ is surjective for all integers $m$, hence $S_2\otimes_k S_{j}\to S_{2+j}$ is also surjective for all even integers $j\ge 2$. On the other hand, the line bundle $\mcl^{\otimes 3}$ on $C$ is Castelnuovo-Mumford $0$-regular (with respect to the very ample line bundle $\omega_C$). To see this, it suffices to note that $\chi(\mcl) = 0$ by Riemann-Roch, and hence
\[
\hh^1(C,\mcl^{\otimes 3}\otimes \omega_C^{-1}) = \hh^1(C,\mcl) = \hh^0(C,\mcl) = 0\,.
\]
By \cite[Theorem 1.8.3]{Lazarsfeld:2004}, it follows that the multiplication map $S_3 \otimes \mathrm{Sym}^m(S_2)\to S_{2m+3}$ is also surjective for all integers $m$, which in turn implies the surjectivity of $S_2\otimes_k S_j\to S_{2+j}$ for all odd integers $j\ge 3$. This verifies the condition (2) from Corollary \ref{co:noU-geometric}, while condition (1) is clear from the choice of $\mcl$.
\end{proof}

\begin{remark}
With $g$ the genus of the curve $C$, a direct computation gives that the Hilbert series of the ring $S$ in Theorem \ref{th:gorenstein} is
\[
\frac{(1 - 2t + (g+1)t^2 - 2t^3 + t^4)}{(1-t)^2}
\]
 If $S$ were a complete intersection, the numerator would be  product of cyclotomic polynomials. One can check that this is not the case.

The ring $S_\fm$ in Theorem~\ref{th:gorenstein} is also not a rational singularity. In fact, rational surface singularities have minimal multiplicities---see \cite[Theorem on page 94]{Reid:1997}---and such rings admit Ulrich modules, see \cite[Proposition 2.5]{Brennan/Herzog/Ulrich:1987}. It is an interesting question whether there exist rational singularities with no Ulrich modules.
\end{remark}

\begin{ack}
Many thanks to Kevin Tucker for helpful conversations concerning this work, and a referee for their suggestions. The authors were partly supported by National Science Foundation  (NSF) grants DMS-2001368 (SBI);  DMS-2302430 (LM); DMS-1952366 (LM); DMS-2200732 (MW); DMS-2240926 (ZZ), and DMS-2234736 (ZZ). Ma was also partly supported by a Sloan fellowship and by a grant from the Institute for Advanced Study School of Mathematics. Zhuang was also partly supported by a Clay research fellowship, as well as a Sloan fellowship. 

This material is based upon work supported by the NSF under Grant No. DMS-1928930 and by the Alfred P. Sloan Foundation under grant G-2021-16778, while the first three authors were in residence at the Simons Laufer Mathematical Sciences Institute (formerly MSRI) in Berkeley, California, during the Spring 2024 semester.  
\end{ack}

\providecommand{\bysame}{\leavevmode\hbox to3em{\hrulefill}\thinspace}
\providecommand{\MR}{\relax\ifhmode\unskip\space\fi MR }
\providecommand{\MRhref}[2]{%
  \href{http://www.ams.org/mathscinet-getitem?mr=#1}{#2}
}
\providecommand{\href}[2]{#2}

\end{document}